\documentclass{amsart}

\usepackage{mathrsfs}
\usepackage{hyperref}
\usepackage{colortbl}
\usepackage{verbatim}
\usepackage{enumitem}

\newtheorem{theorem}{Theorem}[section]
\newtheorem{proposition}[theorem]{Proposition}

\newtheorem{lemma}[theorem]{Lemma}
{\theoremstyle{definition}
\newtheorem{definition}[theorem]{Definition}
\newtheorem{example}[theorem]{Example}
\newtheorem{remark}[theorem]{Remark}
}

\newtheorem*{thma}{Theorem A}
\newtheorem*{thmb}{Theorem B}
\newtheorem*{thmc}{Theorem C}
\newtheorem*{proa}{Proposition A}
\newtheorem*{corA}{Corollary A}
\newtheorem*{corB}{Corollary B}

\newcommand{\N}{\mathbb{N}}
\newcommand{\Z}{\mathbb{Z}}
\newcommand{\R}{\mathbb{R}}
\newcommand{\C}{\mathbb{C}}
\newcommand{\K}{\mathbb{K}}
\newcommand{\V}{\mathcal{V}}
\newcommand{\coff}{\textnormal{coff}}
\newcommand{\tatop}[2]{\genfrac{}{}{0pt}{10}{#1}{#2}}

\begin{document}
\title{Foliations, solvability and global injectivity}

\author[F. Braun, J.R. dos Santos Filho \MakeLowercase{and} M.A. Teixeira]
{Francisco Braun$^*$, Jos\'e Ruidival dos Santos Filho$^\dagger$ \MakeLowercase{and} Marco Antonio Teixeira$^\ddagger$}

\address{$^*,^\dagger$ Departamento de Matem\'atica, Universidade Federal de S\~ao Carlos, 13565-905 S\~ao Carlos, S\~ao Paulo, Brazil}
\email{franciscobraun@dm.ufscar.br}
\email{ruidival@ufscar.br}

\address{$^\ddagger$ Departamento de Matem\'atica, Universidade Estadual de Campinas, 13083-970 Cam\-pi\-nas, S\~ao Paulo, Brazil}
\email{teixeira@ime.unicamp.br}

\subjclass[2010]{Primary 37C10,35F05; Secondary 14R15}
\keywords{Foliations, Global Solvability, Global Injectivity}

\date{\today}

\begin{abstract}
Let $F: \R^n\to\R^n$ be a $C^{\infty}$ map such that $DF(x)$ is invertible for every $x\in\R^n$. 
Although being a local diffeomorphism, $F$ is not necessarily globally injective if $n\geq2$. 
Finding additional assumptions implying the global injectivity of $F$ for $n\geq 2$ is object of intense study in several areas of Mathematics. 
In this paper we revisit some assumptions and relations between them in the bidimensional case and discuss the natural higher dimensional situation. 
\end{abstract}

\maketitle

\section{Introduction}
Let $F: \R^n\to\R^n$ be a $C^{\infty}$ map such that $DF(x)$ is invertible for each $x\in\R^n$. 
From the inverse function theorem, $F$ is a local diffeomorphism but, in general, it is not injective. 
In fact even in the polynomial case in $\R^2$, S. Pinchuk \cite{Pinchuk} proved that additional hypotheses are needed for the injectivity of $F$. 
The same problem concerning polynomial maps in $\mathbb{C}^n$ is yet an open question, widely known as \emph{Jacobian conjecture} (see \cite{BCW} or \cite{Es} for references). 
In the general case, additional conditions to obtain the injectivity of $F$ were established in the literature in varied areas of Mathematics, see \cite{BCW, bruna, braun, Gutierrez1, LX, P, St1, St2}. 
In this paper we discuss relations between some of these additional requirements in different settings. 
Precisely, we relate concepts such as global solvability of vector fields, the non-existence of half-Reeb components of foliations and connectedness of level sets of functions, stressing the differences between the bi-dimensional case and the higher dimensional one. 
In the former case, the discussed assumptions are sufficient for the injectivity, whereas in the later one, this might be not true. 

We begin recalling that in the holomorphic setting, Y. Stein \cite{St1,St2} proposed a relation between global properties of certain vector fields associated to $F$ with the Jacobian conjecture for $n=2$. 
After that, T. Krasi\'nski and S. Spodzieja \cite{KS} improved this result for the $n$-dimensional case. 
(Early connections have been proposed before, see \cite{BCW}). 
Precisely, given a $C^{\infty}$ map $F = (F_1, \ldots, F_n) : \K^n\to \K^n$, with $\K = \R$ or $\C$, we define $n$ vector fields $\V_i$, $i=1,\ldots,n$, as follows:
\begin{equation}\label{logus}
\V_i(\phi)=\det\big( D(F_1, \ldots, F_{i-1}, \phi, F_{i+1}, \ldots, F_n)\big), 
\end{equation} 
where $\phi: \K^n \to \K$ is a $C^\infty$ function. 
A straightforward remark is that $F_j$, $j\in \{1,\ldots, n\}\setminus\{i\}$, are first integrals of $\V_i$, and so each $\V_i$ has $n-1$ first integrals linearly independent. 
Observe also that, for $n = 2$, $\V_i = (-1)^{i}H_{F_j}$, $i\neq j\in\{1,2\}$, where $H_{F_j}$ stands for the Hamiltonian vector field associated to $F_j$, $H_{F_j}=-\partial_2F_j\partial_1+\partial_1F_j\partial_2$ (in this paper, for $f:\K^n\to\K$, we denote $\partial_i f=\partial f/\partial_{x_i}$). 
We patch together the above mentioned results in the following theorem, where $E_n$ stands for the set of complex holomorphic functions on $\C^n$.
\begin{theorem}[\cite{KS, St1, St2}]\label{stein}
Let $F:\C^n\to\C^n$ be a polynomial map such that $\det DF$ is a non-zero constant. 
If  $\V_i(E_n) = E_n$ for $n-1$ different indices $i\in\{1,\ldots, n\}$, then $F$ is injective.
\end{theorem}

In this paper we analyze the real case, so we consider the vector fields given in \eqref{logus} for the \emph{$C^{\infty}$ real} case. 
We first recall the notion of \emph{global solvability} of vector fields. 
Let $M$ be a $C^\infty$ manifold and $C^\infty(M)$ be the set of $C^\infty$ real functions on $M$. 

\begin{definition}
We say that a vector field $\mathcal{X}:C^{\infty}(M)\to C^{\infty}(M)$ is \emph{globally solvable} if $\mathcal{X} \left(C^{\infty}(M)\right)=C^{\infty}(M)$.
\end{definition}

Given $F: \R^2 \to \R^2$ such that $DF(x)$ is invertible for each $x \in \R^2$, it is known that the global solvability of $\V_1$ or $\V_2$ is sufficient for the injectivity of $F$. 
This is a consequence of Proposition \ref{propa} below. 
So in dimension $2$, the natural $C^\infty$ real counterpart of Theorem \ref{stein} is valid. 
But this is not true in higher dimensions. 
Indeed, in \cite{BS} a family of $C^\infty$ non-injective maps $F: \R^3 \to \R^3$, with $DF(x)$ invertible for each $x \in \R^3$ and such that $\V_1$ and $\V_2$ are globally solvable, was constructed. 
A polynomial counterexample is not known. 

On the other hand, another sufficient condition for global injectivity in $\R^2$, now related to foliations, was proposed by Gutierrez in \cite{Gutierrez1}. 
Given a $C^{\infty}$ submersion $f: \R^n \to \R$, we denote by $\mathscr{F}(f)$ the $C^{\infty}$ foliation of dimension $n-1$ whose leaves are the connected components of the level sets of $f$. 
In $\R^2$ the notion of \emph{half-Reeb component}, or hRc, of $\mathscr{F}(f)$, see Definition \ref{vigesima}, was introduced in \cite{Gutierrez1} as an important tool to prove the celebrated Markus-Yamabe conjecture in that paper (see also  \cite{gutierrez} and a variation in \cite{teixeira}). 
Gutierrez result is that a $C^\infty$ map $F: \R^2 \to \R^2$, with $DF(x)$ invertible for all $x \in \R^2$, is injective provided $\mathscr{F}(F_1)$ or $\mathscr{F}(F_2)$ has no hRc, again a consequence of Proposition \ref{propa}. 
In $\R^3$, C. Gutierrez and C. Maquera \cite{maquera} extended the definition of half-Reeb components of $\mathscr{F}(f)$ for submersions $f:\R^3\to\R$, as we will recall in Definition \ref{furgao}. 
They used this in \cite{maquera} to prove a version, in $\R^3$, of a Jacobian conjecture proposed by Z. Jelonek in \cite{J}. 
A result of \cite{maquera} says that $\mathscr{F}(f)$ has no hRc's if and only if the leaves of $\mathscr{F}(f)$ are simply connected. 
Below we will provide explicit examples of injective and non-injective local diffeomorphisms $F = (F_1, F_2, F_3): \R^3 \to \R^3$ such that $\mathscr{F}(F_1)$, $\mathscr{F}(F_2)$ and $\mathscr{F}(F_3)$ have no hRc's. 
That is, as occurring in the sufficient conditions (for global injectivity) related to global solvability, the conditions related to the non-existence of hRc's only work in the $2$-dimensional case. 

We consider yet the following very known condition, now working in any dimension, for global injectivity, using only connectedness of level sets. 
\begin{lemma}\label{connected}
Let $F = (F_1, \ldots, F_n): \R^n\to\R^n$ be a local homeomorphism. 
If for some $i\in\{1,\ldots,n\}$ the intersections 
$$
\bigcap_{\tatop{j=1}{j\neq i}}^n F_j^{-1}(\{c_j\}),\  c_j\in\R 
$$
are connected, then $F$ is injective.
\end{lemma}
\noindent
The proof is simple: the local invertibility of $F$ says that the function $F_i$ must be injective in each of the curves $\bigcap_{\tatop{j=1}{j\neq i}}^n F_j^{-1}(\{c_j\})$, and so the injectivity of $F$ follows from the connectedness assumption. 
See also Lemma \ref{prisma} below. 
Even for $n = 2$ the condition of Lemma \ref{connected} is not necessary for global injectivity, as shows the map $F(x,y) = \big( e^x(1-y^2), e^x(1+y-y^2)\big)$, although $F$ is not surjective (see Corollary A below). 

What happens is that the above three sufficient conditions for global injectivity are equivalent in dimension $2$:  
\begin{proposition}\label{propa}
Let $F = (F_1, F_2):\R^2\to\R^2$ be a $C^{\infty}$ map such that $\det DF$ is nowhere zero. 
Given $i \in \{1, 2\}$, then for $j \neq i$ the following statements are equivalent:
\begin{enumerate}[label={\textnormal{(\alph*)}}]
\item\label{wwww} $F_j^{-1}(\{c\})$ is connected for each $c\in\R$.
\item\label{ww} $\V_i$ is globally solvable.
\item\label{www} $\mathscr{F}{(F_j)}$ has no hRc. 
\end{enumerate}
\end{proposition}
\noindent 
So it readily follows that the two above-mentioned injectivity results in $\R^2$ are direct consequences of Proposition \ref{propa} and Lemma \ref{connected}. 

We remark that Proposition \ref{propa} is not completely new: partial versions of it can be found in \cite{BS, gutierrez, Gutierrez1, teixeira}. 
In this paper, anyway, we will provide a different proof than the existing ones. 
Moreover, recalling that $\V_i = (-1)^{i}H_{F_j}$, the statements 
(a), (b) and (c) 
%\ref{wwww}, \ref{ww} and \ref{www} 
are all about the component $F_j$ of $F$. 
Actually, we do not need that $F_j$ is a component of a local diffeomorphism, only that it is a submersion. 
We will prove: 
\begin{proa}
Let $f:\R^2\to \R$ be a $C^{\infty}$ submersion. 
The following statements are equivalent: 
\begin{enumerate}[label={\textnormal{(\alph*)}}]
\item\label{Pconnected} $f^{-1}(\{c\})$ is connected for each $c\in\R$. 
\item\label{Psolvable} $H_f$ is globally solvable. 
\item\label{PReeb} $\mathscr{F}(f)$ has no half-Reeb components. 
\end{enumerate} 
\end{proa} 
\noindent 
Clearly, Proposition \ref{propa} follows from Proposition A. 
As we will see, the really hard part of the proof is that (c) $\implies$ (a), the other implications are almost simple remarks. 
Our proof of this fact also provides a qualitative result about hRc's in $\R^2$, that as long as we know is not yet registered: 
\begin{quote} 
Given any arc connecting two connected components of a level set of a submersion $f: \R^2 \to \R$, there exists a hRc of $\mathscr{F}(f)$ with non-compact edges intersecting this arc. 
\end{quote}
See the precise statement in Corollary B below. 
 
As far as we know, there do not exist comparisons between global solvability of vector fields with either connectedness of certain intersections as well as with the (non-)existence of hRc's in foliations of $\R^3$. 
Our next two results, to some extent, fill in these gap. 
\begin{thma}\label{thmaA}
Let $n\geq 3$ and $F = (F_1, \ldots, F_n): \R^n\to\R^n$ be a $C^{\infty}$ map such that $\det DF$ is nowhere zero. 
The following statements are true. 
\begin{enumerate}[label={\textnormal{(\alph*)}}]
\item\label{02} Let $i\in\{1,\ldots,n\}$. 
If there are two different indices $i_1,i_2\in\{1,\ldots, n\}\setminus\{i\}$ such that for each leaf $L_{i_k}\in \mathscr{F}( F_{i_k})$, $k = 1, 2$, and for each $c_j\in\R$ the set 
$$
\left(\bigcap_{j\notin \{i,i_k\}} F_j^{-1}(\{c_j\})\right) \cap L_{i_k}
$$
is connected, then $\V_i$ is globally solvable.

\item\label{emR3} Let $n=3$ and $i\in\{1,2,3\}$. 
If $\V_i$ is globally solvable, then for $j,k\in\{1,2,3\}\setminus\{i\}$, $j\neq k$,  $L_{k}\in\mathscr{F}(F_{k})$ and $c \in \R$, the set
$$
F_j^{-1}(\{c\})\cap L_{k}
$$
is connected.
\end{enumerate}
\end{thma}

\begin{remark}
Observe that the connected sets here are not the connected sets of Lemma \ref{connected}, so we do not necessarily get the injectivity of $F$. 
But if the assumption of the lemma is true, namely, if for each $c_j \in \R$ the set $\cap_{\tatop{j\neq i}{j=1}}^n F_j^{-1}(\{c_j\})$ is connected, then the sets of Theorem A are also connected and we get the global solvability of $\V_i$. 
The reciprocal result is not true even in $\R^3$ according to the already mentioned non-injective maps $F: \R^3\to \R^3$ with $\V_1$ and $\V_2$ globally solvable given in \cite{BS}: the non-injectivity of $F$ says that for any $i$, there always exists $c_j\in \R$ such that $\cap_{\tatop{j\neq i}{j=1}}^n F_j^{-1}(\{c_j\})$ is not connected. 
Another much simpler example is $F(x)=\left(e^{x_1}\cos x_2,e^{x_1}\sin x_2,x_3,\ldots,x_n\right)$. 
Here we have $\det DF(x)=e^{2x_1}>0$ for all $x\in\R^n$. 
Moreover, for each $i = 3,\ldots,n$, it is clear that $\V_i = e^{2x_1}\partial_i$ is a globally solvable vector field. 
On the other hand, taking $c_1 = 1$, $c_2 = 0$ and $c_3,\ldots,c_n\in\R$, it follows that for each $i = 3, \ldots, n$, the set 
$$
\bigcap_{j\neq i}F_j^{-1}(\{c_j\}) = \bigcup_{k\in \Z}\left\{(0,2k\pi,c_3,\ldots,c_{i-1},z,c_{i+1},\ldots,c_n)\ |\ z\in\R\right\} 
$$
is disconnected. 
The vector fields $\V_1$ and $\V_2$ are not globally solvable, as we will see in Example \ref{appl}. 
\end{remark}

In order to prove statement (b) of Theorem A we will use the next result, a comparison between the concept of hRc, according to Definition \ref{furgao}, with the solvability one.  

\begin{thmb}\label{thmB}
Let $F = (F_1, F_2, F_3):\R^3\to\R^3$ be a $C^{\infty}$ map such  that $\det DF$ is nowhere zero. 
If $\mathscr{F}(F_i)$ has a hRc, then $\V_j$ is \emph{not} globally solvable for $j\in\{1,2,3\}\backslash\{i\}$. 
\end{thmb}
\noindent
The converse of Theorem B fails in general. 
Indeed, take for instance the already considered map $F(x) = \left(e^{x_1}\cos x_2,e^{x_1}\sin x_2,x_3\right)$ for $n=3$. 
Since each leaf of the foliation $\mathscr{F}(F_i)$ is homeomorphic to $\R^2$, it follows that $\mathscr{F}(F_i)$ cannot have half-Reeb components, for $i = 1, 2, 3$. 
But as we have already mentioned, neither $\V_1$ nor $\V_2$ are globally solvable (Example \ref{appl}). 
\begin{remark}
Since the map $F$ in the last example is not injective, we see that the non-existence of hRc's on the foliations $\mathscr{F}(F_i)$ does not guarantee global injectivity.
\end{remark}

\begin{remark}\label{cof}
In general, the existence of hRc does not interfere on the injectivity of $F$, for instance, let $F_1: \R^3\to\R$ be given by $F_1(x) = x_1^2 + x_2^2 - e^{x_3}$. 
Clearly $F_1$ is a submersion and it is simple to prove that the set $\{x\in\R^3\ |\ x_1^2+x_2^2\leq e^{x_3},\ x_3\leq 0\}$ is a half-Reeb component of $\mathscr{F}(F_1)$ with compact face $\{(x_1,x_2,0)\ |\ x_1^2+x_2^2\leq 1\}$ and $F_0(S^1)=\{(x_1,x_2,0)\ |\ x_1^2+x_2^2=1\}$ (contained in the leaf $x_1^2 + x_2^2 = e^{x_3}$), see Definition \ref{furgao}. 
Now, let $G = (F_2,F_3): \R^2\to\R^2$ be a $C^{\infty}$ map such that $\det DG$ is nowhere zero and define $F(x) = \big(F_1(x), G(x_1,x_2)\big)$. 
It follows that $\det DF = -e^{x^3} \det DG$, and we can have two situations:
\begin{enumerate}[label={\textnormal{(\roman*)}}] 
\item\label{hhhhh} F is not injective if $G$ is not injective. 
\item\label{hhhh} $F$ is injective if $G$ is injective.
\end{enumerate}
Take $G=\left(e^{x_1} \cos x_2, e^{x_1} \sin x_2\right)$, for item (i) and $G = (x_1, x_2)$ for item (ii), for instance. 

A final comment here is that by taking this last $F = (F_1, x_1, x_2)$ and the linear transformation $L = (x_1, x_1+x_2, x_1+x_3)$, and then considering the global injective map $H = L\circ F = (F_1, x_1 + F_1, x_2 + F_1)$, it follows that $H_2^{-1}(0)$ and $H_3^{-1}(0)$ are not simply connected and so the three foliations $\mathscr{F}(H_i)$, $i=1,2,3$, will have hRc's. 
\end{remark}

On the other hand, for any $i$, global solvability of $\V_i$, non-existence of hRc's in $\mathscr{F}(F_i)$ or connectedness of the intersections in Lemma \ref{connected} are necessary conditions for global \emph{bijectivity} of local diffeomorphisms. 
Precisely: 
\begin{corA}
Let $F: \R^n \to \R^n$ be a global diffeomorphism, then for every $i=1,\ldots, n$ the vector field $\V_i$ is globally solvable, the intersections $\bigcap_{\tatop{j=1}{j\neq i}}^n F_j^{-1}(\{c_j\}),\  c_j\in\R$, are connected and, if $n\leq 3$, then  $\mathscr{F}(F_i)$ has no hRc's. 
\end{corA}
This follows because $\bigcap_{\tatop{j=1}{j\neq i}}^n F_j^{-1}(\{c_j\}) = F^{-1}(\{c_1\} \times \ldots \{c_{i-1}\} \times \R \times \{c_{i+1}\} \times \cdots \times \{c_n\})$ is connected for any $i$. 
Then case $\R^2$ follows from Proposition A. 
The higher dimension case follows from Theorem A (see the remark right after this theorem), and from Theorem B. 

Therefore, the existence of a non global solvable vector field $\V_i$ or the existence of a hRc in $\mathscr{F}(F_i)$ are obstructions for $F$ to be a global diffeomorphism. 
In particular, in the polynomial case, where injectivity implies surjectivity, according to \cite{Birula}, these are obstructions for global injectivity. 

We finish the paper with a discussion about local invertibility. 
In Remark \ref{cof}, a local invertible map $F = (F_1, F_2, F_3): \R^3\to \R^3$ was constructed by ``completing'' a given submersion $F_1$ with coordinate functions $F_2$ and $F_3$. 
This is not always possible. 
In dimension $2$, the following result from \cite{bruna} provides an obstruction for a given submersion $f: \R^2 \to \R$ to be completed to a local invertible map $(f,g): \R^2 \to \R^2$. 
It turns out that the obstruction is related to the hRc concept. 
Recall that $H_f(g) = \det D(f,g)$. 
\begin{theorem}[\cite{bruna}]\label{eymael}
Let $f:\R^2\to\R$ be a $C^{\infty}$ submersion, $\mathcal{A}\subset\R^2$ be a hRc of $\mathscr{F}(f)$ and $U$ be a neighborhood of $\mathcal{A}$. 
If $h: U\to[0,\infty)$ is a $C^{\infty}$ function such that 
$$
\int_{\mathcal{A}}h=\infty,
$$
then there exists no $C^{\infty}$ function $g:U\to\R$ such that $H_fg = h$ in $U$. 
\end{theorem}
In \cite{bruna,braun} for instance, Theorem \ref{eymael} was used to show that a polynomial submersion $p: \R^2 \to \R$ of degree less than or equal to $4$ cannot be a component of a \emph{non-injective} polynomial local diffeomorphism $(p, q): \R^2 \to \R^2$. 
\noindent
In this paper we extend Theorem \ref{eymael} to $\R^n$ in Theorem C. 
In order to state it, we need to define a new concept that we call a \emph{mild half-Reeb component}, or simply \emph{mhRc}, according to Definition \ref{marina}, of a foliation $\mathscr{F}(f)$, where $f:\R^n\to\R$ is a $C^{\infty}$ submersion. 
As we will see, in $\R^2$ and in $\R^3$, half-Reeb components are mild half-Reeb components. 
\begin{thmc}\label{thmC}
Let $f:\R^n\to\R$ be a $C^{\infty}$  submersion, $\mathcal{B}\subset\R^n$ be a mhRc of $\mathscr{F}(f)$ and $U$ be a neighborhood of $\mathcal{B}$. 
If $h: U\to [0,\infty)$ is a $C^{\infty}$ function such that 
$$
\int_{\mathcal{B}}h=\infty,
$$
then there exists no $C^{\infty}$ map $F=(F_1, \cdots, F_n):U\to\R^n$ with $F_1 = f$ and $\det DF = h$ in $U$.
\end{thmc} 

The paper is organized as follows: We prove statement (a) of Theorem A in Section \ref{firstsection}. 
Then we revisit hRc's in $\R^2$ proving Proposition A in Section \ref{revisited}. 
Theorem B and statement (b) of Theorem A are proved in Section \ref{dimensao3}. 
Finally, in Section \ref{finalresult}, we prove Theorem C. 

\section{Global solvability and connected components}\label{firstsection}
We begin this section by gathering together some properties of integral curves of the vector fields $\V_i$ defined in \eqref{logus}. 
Then we recall a characterization of global solvability of a vector field by means of the geometry of its integral curves. 
These will provide the ingredients to the proof of the first part of Theorem A. 
 
\begin{lemma}\label{prisma}
Let $F = (F_1, \ldots, F_n): \R^n\to\R^n$ be a $C^{\infty}$ map such that $\det DF$ is nowhere zero. 
Then for each $i\in\{1,\ldots,n\}$, 
\begin{enumerate}[label={\textnormal{(\alph*)}}] 
\item The integral curves of $\V_i$ are the non empty connected components of 
\begin{equation}\label{logus2}
\bigcap_{\tatop{j=1}{j\neq i}}^n F_j^{-1}(\{c_j\}),\  c_j\in\R.
\end{equation}
\item The function $F_i$ is strictly monotone along the integral curves of $\V_i$.
\item The $\alpha$- and $\omega$-limit sets in $\R^n$ of each integral curve of $\V_i$ are empty.
\end{enumerate}
In particular, for each $i_1<i_2<\cdots<i_k\in\{1,\ldots,n\}$, $k<n$, the intersections $\cap_{j=1}^k F_{i_j}^{-1}(\{c_j\})$, $c_j\in\R$, are empty or unbounded. 
\end{lemma}
\begin{proof}
Applying the definition of $\V_i$, we get $\V_i(F_j)=\delta_{ij}\det DF$, for each $j\in \{1,\ldots, n\}$, where $\delta_{ij}$ stands for the Kronecker delta. 
Thus for a given integral curve $\gamma$ of $\V_{i}$, it follows that $(F_j\circ\gamma)'(t)=\delta_{ij}\det DF(\gamma(t))$ for all $t$ in the maximal interval of solution of $\gamma$. 
This shows that $F_i$ is strictly  monotone along $\gamma$, proving  (b), and that $\gamma$ is contained in a connected component of an intersection of \eqref{logus2}. 
Since these intersections are $1$-dimensional manifolds by the implicit function theorem, the connected component that contains $\gamma$ must coincide with $\gamma$, because  $\V_i$ has no singular points. 
This proves (a). 
To prove (c), we first observe that (a) implies that the integral curves of $\V_i$ are closed sets. 
So each integral curve of $\V_i$ contains its $\alpha$- and $\omega$-limit sets. 
If for an integral curve of $\V_i$ one of these sets was non-empty, we would have a periodic integral curve, a contradiction with (b). 
\end{proof}

The following is a well known result characterizing the global solvability of a given vector field in terms of the geometry of its integral curves. 
It is part of Theorem 6.4.2 of \cite{HormanderA2}. 

\begin{lemma}\label{vectra}
Let $M$ be a $C^{\infty}$ manifold and $\mathcal{X} :C^{\infty}(M)\to C^{\infty}(M)$ be a vector field on $M$. 
Then $\mathcal{X}$ is globally solvable if and only if  
\begin{enumerate}[label={\textnormal{(\alph*)}}]
\item No integral curve of $\mathcal{X}$ is contained in a compact subset of $M$, and
\item\label{Xconvexidade2} For each compact $K\subset M$, there exists a compact $K'\subset M$ such that every compact interval on an integral curve of $\mathcal{X}$ with end points in $K$ is contained in $K'$.
\end{enumerate}
\end{lemma}

\begin{example}\label{appl}
As an application of this two lemmas we prove here that the vector fields $\V_1$ and $\V_2$ rising from the map of the introduction section $F(x) = \left(e^{x_1}\cos x_2,e^{x_1}\sin x_2,x_3, \ldots, x_n\right)$ are not globally solvable. 

For every $c > 0$, the curve $\gamma_c = \{(x_1,x_2,0,\ldots,0)\ |\ e^{x_1}\sin x_2 = c,\ x_2\in(0,\pi)\}$ is an integral curve of $\V_1$, by (a) of Lemma \ref{prisma}. 
Taking the compact set $K = \{0\}\times[0,\pi]\times\{0\}^{n-2}$, it follows that there is no compact set $K'$ as in condition (b) of Lemma \ref{vectra}, and so $\V_1$ is not globally solvable. 
Indeed, observe that $\gamma_c$ can be parametrized as $x_1 = \xi_c(x_2) \dot{=} -\ln \frac{\sin x_2}{c}$, $x_2\in (0,\pi)$. 
For $0 < c < 1$, the function $\xi_c$ has two zeros in $(0,\pi)$ and $\pi/2$ is a global minimum point between them. 
Since $\xi_c(\frac{\pi}{2}) = \ln \frac{1}{c}$ goes to $-\infty$ as $c$ tends to $0^+$, it follows that the interval of $\gamma_c$ with end points in $K$ cannot be contained in a fixed compact set $K'$ for all $c>0$. 

We can use the same idea to prove that $\V_2$ is not globally solvable.
\end{example}

\begin{proof}[Proof of statement \textnormal{(a)} of Theorem \textnormal{A}]
Without loss of generality, we assume that $i = 1$, $i_1 = 2$ and $i_2 = 3$. 
We denote by $\gamma_x$ the integral curve of $\V_1$ through $x\in\R^n$. 
Suppose on the contrary that $\V_1$ is not globally solvable. 
From (c) of Lemma \ref{prisma}, no integral curve of $\V_1$ is contained in a compact subset of $\R^n$. 
Thus from Lemma \ref{vectra}, there exist sequences $\{x_{k}\}\subset \R^n$ and $\{t_k\}$, $\{s_k\}\subset \R$, with $0<s_k<t_k$, and $a, b\in \R^n$ such that 
\begin{equation}\label{594}
x_k\to a,\ \ \ \ \gamma_{x_k}(t_k)\to b,\ \ \ \ \left|\gamma_{x_k}(s_k)\right|>k.
\end{equation} 
We take $\gamma_a$ and $\gamma_b$ the integral curves of $\V_1$ through $a$ and $b$ respectively. 
By using the facts that the level sets of $F_1$ are local transversals to the flow of $\V_1$ and $F_1$ is monotone along each integral curve of $\V_1$ (from (b) of Lemma \ref{prisma}), we conclude from (\ref{594}) and from the flow box theorem that $\gamma_a$ and $\gamma_b$ are different integral curves. 
Hence, from equation \eqref{594}, statement (a) of Lemma \ref{prisma} and the continuity of $F_i$, $i = 2, \ldots n$, it follows that there exist $c_2,\ldots,c_{n}\in\R$ such that $\gamma_a$ and $\gamma_b$ are different connected components of $\cap_{j=2}^{n} F_j^{-1}(\{c_j\})$. 

Therefore, it follows from the hypothesis that $\gamma_a$ and $\gamma_b$ are in different connected components of $F_{3}^{-1}(\{c_{3}\})$. 
Hence $\gamma_a$ and $\gamma_b$ are in two different connected components of $\cap_{j=3}^{n} F_j^{-1}(\{c_j\})$, which we denote by $\Gamma_a$ and $\Gamma_b$, respectively.
Then we take open neighbourhoods $N_a$ and $N_b$ of $a$ and $b$ respectively such that all the leaves of the foliation $\mathscr{F}\left(F_{2}\right)$ crossing $N_a$ intersect $\Gamma_a$ and all the leaves of  $\mathscr{F}\left(F_{2}\right)$ crossing $N_b$ intersect $\Gamma_b$ (this is possible since $\mathscr{F}(F_2)$ is transversal to the foliation given by the connected components of $\cap_{j=3}^n F_j^{-1}(\{c_j\})$). 
We then take $k$ big enough in order that $x_k\in N_a$ and $\gamma_{x_k}(t_k)\in N_b$ and take $L_2\in\mathscr{F}\left(F_{2}\right)$ containing $\gamma_{x_k}$. 
This leaf $L_2$ intersects $\Gamma_a$ and $\Gamma_b$ and thus $\cap_{j=3}^{n} F_j^{-1}(\{c_j\}) \cap L_2$ is disconnected, a contradiction with the hypothesis.
\end{proof}

\noindent
The proof of statement (b) of Theorem A will be given at the end of Section \ref{dimensao3}. 

\section{HRc's in $\R^2$ revisited: proof of Proposition A}\label{revisited}
In this section we recall the definition of a hRc of $\mathscr{F}(f)$ where $f: \R^2 \to \R$ is a $C^{\infty}$ submersion. 
Then we describe this definition in terms of integral curves of $H_f$ and prove Proposition A. 

\begin{definition}[\cite{gutierrez, Gutierrez1}]\label{vigesima}
Let $f:\R^2\to\R$ be a $C^{\infty}$ submersion,  $h_0(x,y)=xy$ and $B=\left\{(x,y)\in[0,2]\times[0,2]\ |\ 0<x+y\leq 2\right\}$. 
We say that $\mathcal{A}\subset \R^2$ is a \emph{half-Reeb component}, or simply a \emph{hRc}, of $\mathscr{F}(f)$ if there exists a homeomorphism $T:B\to\mathcal{A}$ which is a topological equivalence between $\mathscr{F}(h_0)|_{B}$ and $\mathscr{F}(f)|_{\mathcal{A}}$ with the following properties:
\begin{enumerate}[label={\textnormal{(\roman*)}}]
\item The segment $\{(x,y)\in B\ |\ x+y=2\}$ is sent by $T$ onto a transversal section for the leaves of $\mathscr{F}(f)$ in the complement of $T(1,1)$. 
This section is called the \emph{compact edge} of $\mathcal{A}$.
\item Both segments $\{(x,y)\in B\ |\ x=0\}$ and $\{(x,y)\in B\ |\ y=0\}$ are sent by $T$ onto full half-leaves of $\mathscr{F}(f)$. 
These two half-leaves are called the \emph{non-compact edges} of $\mathcal{A}$.
\end{enumerate} 
\end{definition}

It is a known fact that the leaves of $\mathscr{F}(f)$ are the integral curves of $H_f$ (in Lemma \ref{prisma} this is proved when $f$ is a component of a local invertible map $\R^2\to \R^2$). 
Before proving Proposition A, we give an alternative definition of a half-Reeb component $\mathcal{A}$ of $\mathscr{F}(f)$ by using the integral curves of the vector field $H_f$. 
Given $x\in \R^2$, we denote by $\gamma_x(t)$ the integral curve of $H_f$ such that $\gamma_x(0) = x$. 
We denote by $S$ an arc  with end-points $p$ and $q$ in two different integral curves of $H_f$ such that there exists a point $w \in S\backslash \{p, q\}$ with the following  properties: 
\begin{enumerate}[label={\textnormal{(\Roman*)}}]
\item The arcs $S_1 \subset S$, from $p$ to $w$, and $S_2\subset S$, from $q$ to $w$, are transversal sections to the flow of $H_f$ away from the point $w$. 
\item For every point $x$ in $S_1\backslash\{p,w\}$ the integral curve $\gamma_x$ crosses a point $y(x)$ in $S_2\backslash\{w, q\}$. 
\item The map $S_1 \backslash\{p,w\} \ni x \mapsto y(x) \in S_2\backslash\{w,q\}$ is a homeomorphism and extend homeomorphicaly to $S_1$ by setting $y(p) = q$ and $y(w) = w$. 
\end{enumerate}
We let $G$ be the union of the intervals from $x$ to $y(x)$ of $\gamma_x$, for all $x\in S_1\backslash \{p\}$, and define $\mathcal{A} = G\cup (\gamma_p \cap \overline{G}) \cup (\gamma_q \cap \overline{G})$. 

It is clear that a hRc of $\mathscr{F}(f)$ satisfies the properties of the set $\mathcal{A}$ just constructed. 
On the other hand, we \emph{claim that this set $\mathcal{A}$ is a hRc of $\mathscr{F}(f)$ with compact edge being the arc $S$ and non-compact edges the half-solutions $\gamma_p \cap \overline{G}$ and $\gamma_q \cap \overline{G}$.} 

Indeed, consider sequences $\{p_n\}$ and $\{q_n\}$ in $\gamma_p$ and $\gamma_q$, respectively, such that $p_n = \gamma_p(t_n)$ and $q_n = \gamma_q(s_n)$ with $\{t_n\}$ increasing and $\{s_n\}$ decreasing such that  $t_n$ converges to the upper limit of the maximal interval of definition of $\gamma_p$ and $s_n$ converges to the lower limit of the maximal interval of definition of $\gamma_q$. 
Let $T_{p_n}$ and $T_{q_n}$ be transversal sections to the flow through $p_n$ and $q_n$ contained in $\mathcal{A}$, such that $T_{p_n} \cap T_{q_n} = \emptyset$. 

We now construct a decreasing sequence $\{x_n\}$ contained in $S_1$ with the property that for each $n \in \N$ there exists $x_n \in S_1$ such that the integral curve $\gamma_{x_n}$ intersects $T_{p_n}$ and $T_{q_n}$ in points $u_n$ and $v_n$, respectively. 
We denote by $\mathcal{A}_n$ the region bounded by the interval of curves $[p, p_n]$ of $\gamma_p$, $[p_n, u_n]$ of $T_{p_n}$, $[u_n, v_n]$ of $\gamma_{x_n}$, $[v_n, q_n]$ of $T_{q_n}$, $[q_n, q]$ of $\gamma_{q_n}$, and $S$. 

We clearly have $\mathcal{A}_n \subset \mathcal{A}_{n+1}$, $x_n \to p$ and $y(x_n) \to q$. 
We also have $\cup_n \mathcal{A}_n = \mathcal{A}$. 

Now we apply similar construction for the foliation given by the level sets of $h_0(x,y) = x y$ in $B$ (this foliation coincides with the one given by the integral curves of the vector field $H_{h_0} = -x \partial_1 +y \partial_2$ in $B$), i.e. we build a sequence of subsets $\mathcal{A}_n'$ such that $B = \cup_n \mathcal{A}_n'$ (as above, consider sequences $p_n' = (t_n',0)$ and $q_n' = (0, s_n')$ in $(0,2]\times\{0\}$ and $\{0\} \times (0,2]$, respectively, with $t_n'$ and $s_n'$ decreasing to $0$. 
Then take transversal sections $T_{p_n'}$ and $T_{q_n'}$, and consider $x_n'$, $u_n'$, $v_n'$ with similar meaning as above and construct $\mathcal{A}'$ to be the region bounded by the intervals of curves $[(2,0),p_n']$ of $(0,2)\times \{0\}$, $[p_n', u_n']$ of $T_{p_n'}']$, $[u_n', v_n']$ of $\gamma_{x_n'}$, $[v_n', q_n']$ of $T_{q_n'}$, $[q_n', (0,2)]$ of $\{0\} \times (0,2]$, and the segment $x + y = 2$). 

We then define $T_n : \mathcal{A}_n \to B$ a homeomorphism onto the image $\mathcal{A}_n' = T_n (\mathcal{A}_n)$ that carries leaves of $\mathscr{F}(f)|_{\mathcal{A}_n}$ onto leaves of $\mathscr{F}(h_0)|_{\mathcal{A}_n'}$. 
It is clear that we can extend $T_n$ to $T_{n+1} : \mathcal{A}_{n+1} \to B$ with analogous properties. 

We finally define $T: \mathcal{A} \to B$ by $T(z) = T_n(z)$ if $z \in \mathcal{A}_n$. 
By construction it is clear that $T$ is a homeomorphism. 
This proves that this set $\mathcal{A}$ is a hRc of $\mathscr{F}(f)$. 

We also observe that the existence of hRc is equivalent to the existence of inseparable leaves on the foliation $\mathscr{F}(f)$, see \cite{teixeira} for details.

\begin{proof}[Proof of Proposition \textnormal{A}]
In this proof, given $x\in \R^2$, we denote by $I_x$ the maximal interval of solution of $\gamma_x$, where as above $\gamma_x(t)$ is the integral curve of $H_f$ such that $\gamma_x(0) = x$. 
Further, we denote $\gamma_x^+ = \{\gamma_x(t)\ |\ t \geq 0,\ t\in I_x \}$ and $\gamma_x^- = \{\gamma_x(t)\ |\ t \leq 0,\  t\in I_x \}$. 
Finally, given a set $A\subset \R^2$, we say that \emph{$A$ contains the positive end} (respectively, \emph{negative end}) \emph{of $\gamma_x$} if there exists $t_0 \in I_x$,  such that $\gamma_x(t) \in A$ for all $t \geq t_0$ (respectively for all $t \leq t_0$). 

If $H_f$ is not globally solvable, by using Lemma \ref{vectra} and acting analogously as in the first paragraph of the proof of statement (a) of Theorem A, it is simple to conclude that $f$ has a disconnected level set. 
So (a) implies (b). 
  
Also it is clear that (b) implies (c), because the existence of a hRc in the foliation $\mathscr{F}(f)$, according to Lemma \ref{prisma}, would make condition (b) in Lemma \ref{vectra} to fail. 

Thus it remains to prove that (c) implies (a). 
We assume (c) and suppose on the contrary that there exists $c\in\R$ such that $f^{-1}(\{c\})$ is not connected. 
We will construct a hRc of $\mathscr{F}(f)$ prescribing a set $\mathcal{A}$ as explained right after Definition \ref{vigesima}, obtaining a contradiction with (c). 
We take $p$ and $q$ in two distinct connected components of $f^{-1}(\{c\})$ and denote by $\Gamma$ the open connected region of $\R^2$ bounded by the integral curves $\gamma_p$ and $\gamma_q$ (recall statement (a) of Lemma \ref{prisma}). 
Let $\lambda: [0,1]\to\R^2$ be a $C^{\infty}$ injective curve such that $\lambda(0) = p$, $\lambda(1) = q$ and $\lambda((0,1))\subset \Gamma$. 
The curve $\lambda$ separates $\Gamma$ into two open connected regions that we denote by $\Gamma_1$ and $\Gamma_2$, respectively. 

Either the global maximum or the global minimum of $f\circ \lambda(t)$ is attained at a point $t_m \in (0,1)$. 
Hence $\gamma_{\lambda(t_m)}$ is entirely contained in $\Gamma_1 \cup \lambda([0,1])$ or in $\Gamma_2\cup \lambda([0,1])$. 
In particular or $\Gamma_1$ or $\Gamma_2$ contain both the positive and the negative ends of $\gamma_{\lambda(t_m)}$, respectively. 
Therefore, the set 
$$
T = \{t \in (0, 1)\ |\ \textnormal{ both the ends of } \gamma_{\lambda(t)} \textnormal{ are contained in } \Gamma_1 \textnormal{ or in } \Gamma_2 \} 
$$
is non-empty. 
Let $t_1 \in [0,1]$ be the greatest lower bound of $T$. 
We have two possibilities: either $t_1 \in T$ or $t_1 \notin T$. 

In the first possibility, we assume without loss of generality that both the ends of $\gamma_{\lambda(t_1)}$ are in $\Gamma_1$. 
It is clear from the definition of $t_1$ that $t_1 > 0$ and that $\gamma_{\lambda(t_1)}$ does not intersect $\lambda( [0,t_1) )$. 
We take a cross section to the flow of $H_{f}$, $\alpha: [a, b] \to \R^2$, such that $\alpha(a) = \lambda(t_1)$ and such that $\alpha\left( (a,b] \right)$ and $\lambda((0, t_1))$ are contained in the same connected region of $\Gamma$ defined by $\gamma_{\lambda(t_1)}$. 
For $b' \in (a, b)$ close enough to $a$, we conclude that the integral curves of $H_{f}$ crossing $\alpha((a, b'))$ have one end in $\Gamma_1$ and the other end in $\Gamma_2$. 
Without loss of generality, we assume that the negative ends of these curves are in $\Gamma_1$. 
From continuous dependence, by taking $b'$ smaller if necessary, we can assume that for each $r \in (a, b')$, 
$$
s_r = \sup \{s\ |\ \gamma_{\alpha(r)}(s) \in \Gamma_1\cup \lambda \} \in (0,\infty). 
$$
Clearly there is $t(r) \in (0,1)$ such that  
\begin{equation}\label{enters}
\gamma_{\alpha(r)} (s_r) = \lambda(t(r)),\ \ \ \ \gamma_{\alpha(r)} (s) \in \Gamma_2,\ \forall s> s_r,\ s\in I_{\alpha(r)}. 
\end{equation}
The function $(a, b') \ni r\mapsto t(r)$ is clearly injective. 
We \emph{claim that this function is also monotone}. 
The claim follows if we prove that: for each $r_1, r_2, r_3 \in (a,b')$ with $r_1 < r_2 < r_3$, the point $t(r_2)$ is contained in the interval determined by $t(r_1)$ and $t(r_3)$. 
To prove this, consider the unbounded open set $A$ contained in $\Gamma_2$ whose border is the union of the curves $\gamma_{\lambda(t(r_1))}^+$, $\gamma_{\lambda(t(r_3))}^+$ and the interval of $\lambda(t)$ with end points $\lambda(t(r_1))$ and $\lambda(t(r_3))$. 
The curve $\gamma_{\alpha(r_2)}$ is contained in the region bounded by $\gamma_{\alpha(r_1)}$ and $\gamma_{\alpha(r_3)}$. 
In particular, this curve will enter the set $A$ according to \eqref{enters}. 
The only way to do that is crossing the interval of $\lambda(t)$ with end points $\lambda(t(r_1))$ and $\lambda(t(r_3))$. 
This proves that $t(r_2)$ is between $t(r_1)$ and $t(r_3)$. 

From the claim it follows that there exists $t_2 = \lim_{r\to a^+} t(r)$. 
We consider the integral curve $\gamma_{\lambda(t_2)}$. 
We take a cross section $\beta: [c,d] \to \R^2$ with $\beta(c) = \lambda(t_2)$ and $\beta((c, d])$ contained in the open connected region defined by $\gamma_{\lambda(t_2)}$ containing the curve $\gamma_{\lambda(t_1)}$. 
By taking $b'$ and $d$ smaller if necessary, we can assume that $\alpha((a,b'])$ and $\beta((c,d])$ are contained in the open connected region defined by $\gamma_{\lambda(t_1)}$ and $\gamma_{\lambda(t_2)}$. 
For $r_0$ close enough to $a$, we conclude by construction that all the curves $\gamma_{\alpha(r)}$, $r \in (a, r_0]$, intersect $\beta((c,d])$. 
Let $u_0 \in (c, d]$ and $s_0 > 0$ such that $\gamma_{\alpha(r_0)} (s_0) = \beta (u_0)$. 
We take the curve 
$$
S = \alpha([a, r_0]) \cup \gamma_{\alpha(r_0)}([0, s_0]) \cup \beta([c, u_0]). 
$$ 
Using the flow box theorem along the interval of curve $\gamma_{\alpha(r_0)}([0, s_0])$, we can modify the continuous curve $S$ to a $C^{\infty}$ curve $S'$ which is transversal to the flow of $H_{f}$ up to a point $w'$. 
From construction, the arc $S'$, with end-points $p' = \lambda(t_1)$ and $q'= \lambda(t_2)$, satisfy the properties (I), (II) and (III) prescribed after Definition \ref{vigesima}. 
Thus we obtain a hRc of $\mathscr{F}(f)$ with compact edge being $S'$ and non-compact edges being $\gamma_{\lambda(t_1)}^+$ and $\gamma_{\lambda(t_2)}^-$ finishing the proof in case $t_1\in T$. 

Now in the case $t_1 \notin T$, it follows that $\gamma_{\lambda(t_1)}$ has one end in $\Gamma_1$ and one end in $\Gamma_2$ or $t_1 = 0$. 
In both cases, we take a cross section $\alpha: [a,b] \to \R^2$ such that $\alpha(a) = \lambda(t_1)$ and $\alpha((a,b])$ is contained in the open region determined by $\gamma_{\lambda(t_1)}$ containing $q$. 
From the definition of $t_1$ there exists $b' \in (a,b)$ such that each integral curve of $H_{f}$ crossing $\alpha((a,b'))$ has its both ends in $\Gamma_1$ or in $\Gamma_2$ (we use the fact that $\alpha$ is a cross section). 
We assume without loss of generality that they have its both ends in $\Gamma_1$. 
Assuming that the negative end of $\gamma_{\lambda(t_1)}$ is in $\Gamma_1$, it follows from continuous dependence, taking $b'$ smaller if necessary, that for all $r \in (a, b')$, 
$$
s_r = \sup \{s\ |\ \gamma_{\alpha(r)}(s) \in \Gamma_2\cup \lambda \} \in (0,\infty). 
$$ 
As above we define $t(r)$ to be the point of $(0,1)$ such that $\gamma_{\alpha(r)}(s_r) = \lambda(t(r))$. 
Then similar arguments show that $t(r)$ is monotone and as above we construct a hRc with one of its non-compact edges being $\gamma_{\lambda(t_1)}^+$.
\end{proof}

In the above proof that (c) implies (a), we have proven the following qualitative result about the hRc that rises from a disconnected level set of a submersion: 

\begin{corB}\label{quali}
Let $f: \R^2 \to \R$ be a $C^\infty$ submersion. 
Let $\gamma_1$ and $\gamma_2$ be two distinct connected components of a level set $f^{-1}(\{c\})$ and $\Gamma$ be the open connected region whose border is $\gamma_1 \cup \gamma_2$. 
If $\lambda$ is an injective $C^\infty$ curve contained in $\overline{\Gamma}$ connecting $\gamma_1$ and $\gamma_2$, then there exists a hRc of $\mathscr{F}(f)$ contained in $\overline{\Gamma}$ whose non-compact edges intercept the curve $\lambda$. 
\end{corB}

\section{Half-Reeb components and solvability in $\R^3$}\label{dimensao3}

Now we recall the concept of half-Reeb component for a $C^{\infty}$ submersion $f:\R^3\to\R$ as introduced in \cite{maquera}. 
We first define the concept of \emph{vanishing cycle}  for $\mathscr{F}(f)$. 

\begin{definition}\label{furgao0}
Given a $C^{\infty}$ submersion $f:\R^3\to\R$, we say that a $C^{\infty}$ embedding $F_0: S^1\to\R^3$ is a \emph{vanishing cycle} for $\mathscr{F}(f)$ if it satisfies:
\begin{enumerate}[label={\textnormal{(\roman*)}}]
\item\label{1f} $F_0\left(S^1\right)$ is contained in a leaf $L_0$ of $\mathscr{F}(f)$ and is \emph{not} homotopic to a point in $L_0$.
\item $F_0$ can be extended to a $C^{\infty}$ embedding $F:[-1,2]\times S^1\to\R^3$ such that for all $t\in(0,1]$ there is a $2$-disc $D_t$ contained in a leaf $L_t$ with $\partial D_t=F\left(\{t\}\times S^1\right)$.
\item For all $x\in S^1$ the curve $t\in[-1,2]\mapsto F(t,x)$ is transverse to the foliation $\mathscr{F}(f)$, and for all $t\in(0,1]$ the disc $D_t$ depends continuously on $t$.
\end{enumerate}
We say that the leaf $L_0$ \emph{supports} the vanishing cycle $F_0$ and that the map $F$ is \emph{associated} to $F_0$.
\end{definition}

\begin{definition}\label{furgao}
The \emph{half-Reeb component of $\mathscr{F}(f)$ associated to the vanishing cycle $F_0$} is the region
$$
\mathcal{A}=\left(\cup_{t\in(0,1]}D_t\right)\cup L\cup F_0\left(S^1\right),
$$
where $L$ is the  connected component  of $L_0\backslash F_0\left( S^1\right)$ contained in the closure of  $\cup_{t\in(0,1]}D_t$.\\
We say that the transversal  $F\left([0,1]\times S^1\right)$ (to the foliation $\mathscr{F}(f)$) is the \emph{compact face} of $\mathcal{A}$ and $L\cup F_0\left(S^1\right)$ is the \emph{non-compact face} of $\mathcal{A}$. 
\end{definition}
Throughout the paper we simply say that $\mathcal{A}$ is a half-Reeb component of $\mathscr{F}(f)$, which we abbreviate as hRc. 

Observe that a hRc of a $2$-dimensional foliation of $\R^3$ is in some sense a ``rotation'' of a hRc of a $1$-dimensional foliation of $\R^2$. 
Next example motivates this idea. 

\begin{example}
Let $g:\R^2\to\R$ be defined by $g(x,y)=x(1-xy^2)$. 
It is easy to see that $g$ is a submersion and that the set $\{(x,y)\ |\ -1/\sqrt{x}\leq y\leq 1/\sqrt{x}, \ 0<x\leq 1\}$  is a hRc of $\mathscr{F}(g)$ with compact edge being the segment $\{(1,y)\ |\ y\in [-1,1]\}$ and non-compact edges being the curves $y=1/\sqrt{x}$ and $y=-1/\sqrt{x}$, with $x\in (0,1)$.  

Now we rotate this function around the $x$-axis and obtain a hRc in $\R^3$.  
Precisely, we define $f:\R^3\to\R$ by $f(x,y,z)=g(x,y^2+z^2)=x-x^2(y^2+z^2)^2$. 
It is easy to see that $f$ is a submersion and that the set $\{(x,y,z)\ |\ y^2+z^2\leq 1/\sqrt{x},\ 0<x\leq1\}$ is a hRc of $\mathscr{F}(f)$, now according to Definition \ref{furgao}. 
Here the compact face is the set $\{(1,y,z)\ |\ y^2+z^2\leq 1\}$ and the non-compact face is the set  $\{(x,y,z)\ |\ y^2+z^2 = 1/\sqrt{x}, \ 0 < x \leq 1\}$.    
\end{example}

\begin{lemma}\label{unbounded}
Let $\mathcal{A}$ be a hRc according to Definition \textnormal{\ref{furgao}}. 
Then $\textnormal{int} (\mathcal{A})$ is unbounded.
\end{lemma}
\begin{proof}
Suppose on the contrary that $\textnormal{int} (\mathcal{A})$ is bounded. 
So $\overline{L}$ is compact, and hence we can cover it by a tubular neighborhood $U$ and construct $\pi:U\to \overline{L}$ a $C^{\infty}$ submersion with the property that $\pi(w)=w$ and $\pi^{-1}(\{w\})$ is transversal to $\mathscr{F}(f)$ for each $w\in \overline{L}$ (see Lemma 2 of Chapter 4 of  \cite{Camacho}). 
It is clear that this submersion can be taken so that $\pi^{-1}\left(F_0\left(S^1\right)\right)\cap D_t=\partial D_t$ for each $t\in(0,\delta)$, for some $\delta>0$. 

Using continuity, we can choose $t\in(0,\delta)$ small enough so that $D_t\subset U$. 
We fix $p_0\in L$ and define $p_t$ to be the unique point of intersection of $\pi^{-1}(\{p_0\})$ with $D_t$. 
Since $D_t$ is contractible, there exists a $C^{\infty}$ map $H_t:S^1\times[0,1]\to D_t$ such that $H_t(x,0) = F(t,x)$ and $H_t(x,1) = p_t$ for all $x\in S^1$. 
We define $H:S^1\times [0,1]\to \overline{L}$ by $H(x,s) = \pi\circ H_t(x,s)$. 
This is a $C^{\infty}$ map such that $H(x,0) = F(0,x)$ and $H(x,1) = p_0$ for all $x\in S^1$, which proves that $F_0\left(S^1\right)$ is homotopic to $p_0$, a contradiction with (i) of Definition \ref{furgao0}.   
\end{proof}

\begin{proof}[Proof of Theorem \textnormal{B}]
Let $\mathcal{A}$ be a hRc of $\mathscr{F}(F_i)$. 
We will use the notation of definitions \ref{furgao0} and \ref{furgao}. 
Let $K = F\left([0,2]\times S^1\right)$, where $F_0: S^1\to\R^3$ is a vanishing cycle for the foliation $\mathscr{F}(F_i)$. 
For each $n\in\N$, let $x_n\in\textnormal{int}\left(\mathcal{A}\right)$ such that $x_n\notin B(0,n)$ (this is possible from Lemma \ref{unbounded}). 
For $n$ sufficiently large, there is $t_n\in (0,2]$ such that $x_n\in D_{t_n}$. 
Let $j \in\{1,2,3\}$, $j \neq i$, and consider $\gamma_{x_n}$ the integral curve of $\V_j$ passing through $x_n$. 
From (a) of Lemma \ref{prisma}, $\gamma_{x_n}\subset L_{t_n}$. 
Since $D_{t_n}$ is bounded, statement (c) of the same lemma asserts the existence of $s_n^1, s_n^2$, $s_n^1 < 0 < s_n^2$, in the interval of definition of $\gamma_{x_n}$, such that 
$$
\gamma_{x_n}(s_n^{\tau}) \in \partial D_{t_n}\subset K,\  \textrm{for } \tau = 1,2.
$$
This shows that $\V_j$ is not globally solvable by Lemma \ref{vectra}. 
\end{proof}

\begin{proof}[Proof of statement \textnormal{(b)} of Theorem \textnormal{A}]
Without loss of generality we assume that $i=1$, $k=2$ and $j=3$. 
From Theorem B it follows that $\mathscr{F}(F_2)$ does not have hRc. 
From Proposition 2.2 of \cite{maquera} it follows that all the leaves of $\mathscr{F}(F_2)$ are diffeomophic to $\R^2$. 

We suppose on the contrary that there are $c_3\in\R$ and $L_2\in\mathscr{F}(F_2)$ such that $F_3^{-1}(\{c_3\})\cap L_2$ is not connected.
We consider a diffeomorphism $H: L_2\to\R^2$ and the map $G = (G_1, G_2): \R^2\to\R^2$ defined by $G(x) = (F_1, F_3)\circ H^{-1}(x)$. 
It follows that $\det DG(x)\neq 0$ in $\R^2$, because $F_1$ and $F_3$ are transversal to $L_2$, and that $G_2^{-1}(\{c_3\})$ is not connected.  
From Proposition \ref{propa}, the foliation $\mathscr{F}(G_2)$ of $\R^2$ has a hRc. 
Thus going back to $L_2$, we have, in the light of Lemma \ref{prisma}, integral curves of $\V_1$  violating item (b) of Lemma \ref{vectra}, a contradiction with the global solvability of $\V_1$.
\end{proof}

\section{A necessary condition for local invertibility}\label{finalresult}
In this section we prove Theorem C. 

We begin defining the concept of mild half-Reeb component. 

\begin{definition}\label{marina}
Let $f:\R^n\to\R$ be a $C^{\infty}$ submersion. 
We say that a region $\mathcal{B}\subset\R^n$ is a \emph{mild half-Reeb component}, or simply a \emph{mhRc}, of the foliation $\mathscr{F}(f)$ if it satisfies 
$$
\mathcal{B}=\overline{\bigcup_{k\in \N}P_k},
$$
where $P_k$ is a bounded connected open set whose boundary is $C^1$ by parts and given by the union of a subset $Q_k$ of a leaf of $\mathscr{F}(f)$ and a hypersurface $L_k$ of $\R^n$. 
Furthermore, $P_{k}\subset P_{k+1}$ for all $k\in\N$ and there exists a \emph{bounded} hypersurface $L$ of $\R^n$ such that $L_k\subset L_{k+1}\subset L$ for all $k\in\N$.
\end{definition}

\begin{example}
Half-Reeb components in $\R^2$ and $\R^3$ are mild half-Reeb components. 
\end{example}

To prove Theorem C, we will use the vector fields $\V_i$ calculated in coordinates:
\begin{equation}\label{lula}
\V_{i} = \sum_{j=1}^n\coff (a_{ij})\partial_j,
\end{equation}
where $\coff(a_{ij})$ is the cofactor of the entry $a_{ij}$ in the matrix $DF$.

We need the following technical result.

\begin{lemma}\label{ghjkk}
Let $F = (F_1,\ldots, F_n): \R^n\to\R^n$ be a $C^{\infty}$ map. 
Then 
$$
\textnormal{div}(F_i\V_{i}) = \det DF, 
$$
for $i =1,\ldots,n$. 
\end{lemma}

\begin{proof}
If we denote $\V_{F,i}=\V_i$ to emphasize the dependence of the map $F$, it is easy to see that $\V_{F,i}=-\V_{\overline{F},n}$, where $\overline{F}$ is the map $F$ with $F_i$ permuted with $F_n$. 
So, since  $\det(DF)=-\det(D\overline{F})$, it is enough to prove that $\textnormal{div}(F_n\V_n)=\det DF$. 
By \eqref{lula} we have 
$$
\textnormal{div}(F_n\V_{n})=\sum_{k=1}^n(\partial_k F_n) \coff(a_{nk}) + F_n\sum_{k=1}^n\partial_k \coff(a_{nk}).
$$
It is clear that the first term above is $\det DF$. 
We assert that the second term is zero. 
Indeed, we have
$$
\sum_{k=1}^n\partial_k \coff(a_{nk})=\sum_{k=1}^n\sum_{\tatop{j=1}{j\neq k}}^nS_{kj},
$$
where
$$
S_{kj}=(-1)^{n+k}\det\left(\partial_1 g, \ldots , \partial_{j-1} g, \partial_k\partial_j g, \partial_{j+1} g, \ldots, \widehat{\partial_k g}  , \ldots ,  \partial_ng\right)
$$
and $g = \left(F_1, \ldots, F_{n-1}\right)^t$, see \eqref{lula}. 
So it is enough to prove that $S_{kj}=-S_{jk}$. 

Suppose first that $k<j$. 
We have 
\begin{align*} 
S_{kj} & = (-1)^{n+k}\det\left(\partial_1g , \ldots , \partial_{k-1}g , \partial_{k+1}g , \ldots , \partial_{j-1}g , \partial_k\partial_jg , \partial_{j+1}g , \ldots , \partial_ng\right)\\
& = (-1)^{n+k}(-1)^{j-1-k}\det\left(\ldots , \partial_{k-1}g , \partial_k\partial_jg , \partial_{k+1}g , \ldots , \partial_{j-1}g , \partial_{j+1}g , \ldots\right)\\
& = - S_{jk}.
\end{align*}
The case $k>j$ is similar.
\end{proof}

\begin{proof}[Proof of Theorem \textnormal{C}]  
Suppose on the contrary that there are $F_2, \ldots, F_n$ with the required properties. 
We \emph{claim that there exists $M\in\R$ such that}
$$
\int_{P_k}h\leq M,\ \textrm{ for all } k\in \N.
$$
This will be a contradiction with the hypothesis by the monotone convergence theorem. 
Therefore it remains to prove the claim. 
By the divergence theorem 
$$
\int_{P_k}\textnormal{div}(F_n\V_{n})=\int_{Q_k}\left< F_n\V_{n},N \right> w_{Q_k}+\int_{L_k}\left<F_n\V_{n},N\right>w_{L_k},
$$
where $w_{Q_k}$ and $w_{L_k}$ are the volume forms of $Q_k$ and $L_k$, respectively, and $N: Q_k\cup L_k\to\R^3$ is the normal vector field of $Q_k\cup L_k$. 
Since $\V_{n}(F_1) = 0$ and $Q_k$, for $k \in \N$, are parts of level sets of $F_1$, it follows that $\V_{n}$ is tangent to $Q_k$ for each $k\in\N$. 
This means that the sum above has just the second term, which is clearly bounded by 
$$
M=\int_{L}\left| \left< F_n\V_{n}, N\right> \right| w_{L},
$$
where $L$ is the bounded hypersurface given by Definition \ref{marina}, and $w_L$ is its volume form. 
From Lemma \ref{ghjkk} it follows that $\int_{P_k} \det DF \leq M$. 
Then the claim is proved, finishing the proof of the theorem. 
\end{proof}

\section*{Acknowledgements} 
We thank the referee for important comments that improved the presentation of the paper. 
The first author was partially supported by FAPESP grants 10/11323-7, 2014/ 26149-3 and 2017/00136-0. 
The second and third author were partially supported by FAPESP grants 07/08231-0 and 07/06896-5, respectively.

\end{document}